\documentclass[12pt]{article}

\usepackage[utf8]{inputenc}
\usepackage{amsfonts}
\usepackage{amssymb}
\usepackage{amsmath}
\usepackage{amsthm}
\usepackage{fullpage}
\usepackage{graphicx}
\usepackage{color}

\theoremstyle{plain}
\newtheorem{theorem}{Theorem}[section]
\newtheorem{proposition}[theorem]{Proposition}
\newtheorem{corollary}[theorem]{Corollary}
\newtheorem{lemma}[theorem]{Lemma}
\theoremstyle{definition}
\newtheorem{definition}{Definition}
\newtheorem{example}{Example}
\theoremstyle{remark}
\newtheorem{remark}{Remark}
\newtheoremstyle{cond}
  {3pt}			
  {3pt}			
  {}			
  {}			
  {\bfseries}		
  {.}			
  {.5em}		
  {Condition \thmnote{#3}}	
\theoremstyle{cond}
\newtheorem*{condition}{}

\newcommand{\ie}{{\em i.e.}}
\newcommand{\eg}{{\em e.g.}}

\newcommand{\R}{\mathbb{R}}	
\newcommand{\C}{\mathbb{C} }    
\newcommand{\Q}{\mathcal{Q}}
\newcommand{\RS}{\hat{\C}}      

\newcommand{\inte}{\text{\rm Int}\,}

\newcommand{\td}{\tilde{d}}

\newcommand{\diam}{\text{\rm diam}}
\newcommand{\dash}{\text{-}}

\newcommand{\MZ}{\setminus\{0\}} 

\begin{document}

\title{Projective metrics and contraction principles for complex cones}
\author{Loïc Dubois\\ \\
	Department of Mathematics,\\
	University of Cergy-Pontoise,\\
	2 avenue Adolphe Chauvin,\\
	95302 Cergy-Pontoise Cedex, France.}
\date{\today}

\maketitle
\abstract{
In this article, we consider linearly convex complex cones in complex Banach spaces
and we define a new projective metric on these cones. Compared to the hyperbolic
 gauge of Rugh, it has the advantage of being explicit,
and easier to estimate.
We prove that this metric also satisfies a contraction principle like Birkhoff's theorem for the
Hilbert metric. We are thus able to improve existing results on spectral gaps for complex matrices.
Finally, we compare the contraction principles for the hyperbolic gauge and our metric on particular cones,
including complexification of Birkhoff cones. It appears that the contraction principles for our metric and
the hyperbolic gauge occur simultaneously on these cones. However, we get better contraction rates
with our metric.}

{AMS Subject classification codes (2000): 15A48, 47A75, 47B65. }

\section{Introduction}

In his article \cite{Bir57} (see also \cite{Bir67}), Birkhoff proved that the Hilbert projective metric on
convex cones satisfies a contraction principle. He showed that a linear map $T$ preserving a cone
is a contraction for the Hilbert metric; and that this contraction is strict and uniform when the image of
the cone is of finite diameter. He used this idea to prove various theorems on positive operators. This
technique of projective metrics permits to avoid the use of the Leray-Schauder fixpoint theorem as in
\cite{Kr48}. It is typically useful when the operator one considers is not compact, and was
extensively used in thermodynamic formalism (see for instance \cite{FS79}, \cite{FS88} and \cite{Liv95}).

Recently (see \cite{Rugh07}), Rugh extended the contraction principle of Birkhoff to complex cones
in complex Banach spaces. He
introduced a new projective gauge $d_C$ on complex cones
and gave a very general contraction principle.
He then proved various extensions of existing results on positive operators to a complex setting. Also, he
had to abandon the convexity assumption on the cone which is a very useful condition in the real case. As
a consequence, the complex gauge does not satisfy the triangular inequality in
general, and there is not anymore a notion of dual cone.

As a substitute for convexity, there is the natural notion of {\em linear convexity} in the complex setup.
The {\em dual complement} of a linearly convex complex cone then replaces the real dual cone.
These notions originally come from several complex variables analysis
and were introduced for open sets in $\C^n$ by
Martineau in \cite{Mar66} (see also \cite{An04} for more on the subject and historical notes). A complex
cone is said to be linearly convex if through each point in its complement, there passes a complex
hyperplane not intersecting the cone; and then, the dual complement is the set of all linear functionals not
vanishing on the cone. The central idea of our paper is to use this notion of duality to study complex cones.

In the present paper, we define first (in section \ref{sect-CP}) a new projective gauge $\delta_C$
analogously to the Hilbert metric
(compare with formulas (\ref{eq-hilbert}) and (\ref{eq-dual-hilbert}) for the Hilbert metric).
Let $C$ be a complex cone in a complex Banach space and
let $x$ and $y$ be independent vectors in $C$. Then
$\delta_C(x,y)$ is defined to
be $\log(b/a)$, where $b$ and $a$ are respectively the supremum and the infimum of the modulus of
\begin{equation} \label{eq-Exy}
E_C(x,y) = \{z\in \C:\:zx-y\notin C\}.
\end{equation}
$\delta_C$ is actually a projective metric (meaning that it satisfies the triangular inequality) for
a large family of cones, including for instance linearly convex cones
and the cones of section \ref{sect-complexcones}.
For general complex cones, we prove also that $\delta_C$ satisfies a contraction principle
similar to Birkhoff's theorem.

We study more precisely the case of complex matrices in section \ref{sect-CPN}.
First of all, recall that a real matrix $A$ satisfies $A(\R_+^n\MZ)\subset \inte\,\R_+^n$ if and
only if all its entries are positive.
Besides, the Perron-Frobenius theorem claims that in such a situation, the
matrix $A$ has a spectral gap. In other words, there exists a unique eigenvalue $\lambda_m$ of maximal
modulus, which is simple, and the others are of modulus
not greater than $c|\lambda_m|$, $c<1$. Moreover, a positive matrix
$A$ contracts strictly and uniformly the Hilbert metric of $\R_+^n$. This gives an estimate of the
`size' $c$ of the spectral gap: one may
take the contraction coefficient of the Hilbert metric $c=\tanh(\Delta/4)<1$ given by Birkhoff's theorem.
Here $\Delta=\sup_{x,y \in \R_+^n\setminus\{0\}} h_{\R_+^n} (Ax,Ay)<\infty$ is the diameter of $A(\R_+^n\MZ)$
with respect to the Hilbert metric $h_{\R_+^n}$.

This result is generalized in \cite{Rugh07}: a natural extension $\C_+^n\subset \C^n$ of
$\R_+^n$ is defined;  and it is proved (among other things) that a complex matrix $A$ such that
\begin{equation} \label{eq-acpn}
A(\C_+^n\MZ)\subset\inte\,\C_+^n
\end{equation}
has a spectral gap. Using some kind of perturbation argument, Rugh proves also that the
complex matrices $A=(a_{ij})$ such that $|\Im(a_{ij}\overline{a_{kl}})| <
\alpha\leq \Re(a_{ij}\overline{a_{kl}})$ for all indices are examples of matrices satisfying (\ref{eq-acpn}).
 This perturbation technique gives good estimates of the size
of the spectral gap only when the matrix $A$ is close to a positive matrix.

Here, we study the complex cone of {\em all} complex matrices satisfying (\ref{eq-acpn}).
We give a simple condition on the coefficients $a_{ij}$ that characterizes these matrices,
and we provide a sharp estimate of the size of their spectral gap.
We summerize our results on complex matrices in the following theorem.
\begin{theorem}
The complex matrices $A$ satisfying $A(\C_+^n\MZ)\subset \inte\,\C_+^n$
are exactly the matrices satisfying for all indices
$$ \Re(\overline{a_{kp}} a_{lq} + \overline{a_{kq}} a_{lp} ) > | a_{kp} a_{lq} - a_{kq} a_{lp} |.
$$ These matrices have a spectral gap.
If $\lambda_m$ is the leading eigenvalue, the other eigenvalues are of modulus not greater than
$c|\lambda_m|$. The size $c$ of the spectral gap is given by $c=\tanh(\Delta/4)<1$, where $\Delta$ is the
diameter of $A(\C_+^n\MZ)$ with respect to our projective metric $\delta_{\C_+^n}$. We get also the
following simple estimate for $\Delta$. If $\theta\in(0,1)$ and
$\sigma>1$ are such that
\begin{eqnarray*}
 \Re(\overline{a_{kp}} a_{lq} + \overline{a_{kq}} a_{lp} ) &>&\frac{1}{\theta} | a_{kp} a_{lq} - a_{kq}
a_{lp} |, \\
|a_{kp}a_{lq}|&\leq& \sigma^2 |a_{kq}a_{lp}|.
\end{eqnarray*} then
$$ \delta\dash\diam\,A(\C_+^n\MZ) \leq 8 \log \frac{1+\theta}{1-\theta} + 2\log \sigma. $$
\end{theorem}

Finally, in section \ref{sect-comp} we compare the complex gauge $d_C$, and the metric $\delta_C$.
Recall the definition of the gauge $d_C$.
Let $C$ be a complex cone and $x$, $y\in C$ linearly independent. If $0$ and $\infty$ belong to the same
connected component $U$ of the interior $\mathring{L}(x,y)$ of
\begin{equation} \label{eq-intro1}
 L(x,y)=\{z\in\RS:\: zx-y\in C\} \subset \RS,
\end{equation} 
then $d_C(x,y)$ is the Poincaré distance in $U$ between $0$ and $\infty$. Otherwise, $d_C(x,y)=\infty$, and 
$d_C(x,\alpha x)$ is defined to be $0$.
Note that $\delta_C(x,y)<\infty$ only requires that $0$, $\infty \in \mathring{L}(x,y)$ but not necessarily in
the same connected component of $\mathring{L}(x,y)$.
This fact makes the metric $\delta_C$ easier to estimate: one does not need to study the full
geometry of the set $E_C(x,y)$ as we do in section \ref{sect-comp}
but only to estimate $\sup|E_C(x,y)|$ and $\inf|E_C(x,y)|$. 
The use of $\delta_C$ also avoids some technical complications in proofs, see Example \ref{example1}.

To make comparison between $d_C$ and $\delta_C$ possible, we restrict ourselves to the
complex cones of section \ref{sect-complexcones} (including canonical complexifications
of real Birkhoff cones defined in \cite{Rugh07} and thus $\C_+^n$). On these cones,
the geometric configuration of the sets $\mathring{L}(x,y)$
is quite simple: it is simply connected, and its complement
is a union of disks and half-planes (a finite union in the case of $\C_+^n$). We are thus able to prove
the following inequalities.
\begin{eqnarray}
 \frac{1}{2} \delta_C(x,y) &\leq & d_C(x,y), \nonumber \\
 d_C(x,y) &\leq& \pi\sqrt{2}\exp(\delta_C(x,y)/2), \label{eq-middle} \\
 d_C(x,y) &\leq& 3\delta_C(x,y) \textrm{ if  } \delta_C(x,y)< \delta_0. \nonumber
\end{eqnarray}
Here, $\delta_0>0$ does not depend on the cone. Though the constants are probably not optimal, the $\exp$ is
necessary in $(\ref{eq-middle})$.
So the $\delta_C$-diameter is in general significantly smaller than
the $d_C$-diameter and thus gives a better estimate of spectral gaps size (see Remark \ref{remark-estimate}).
However, this proves also that the condition of being of finite diameter does not
depend on whether we use $\delta_C$ or $d_C$.
This means that the contraction principle in \cite{Rugh07} and the one we prove for $\delta_C$
occurs simultaneously for our family of cones.

The gauge $d_C$ does not satisfy the triangular inequality (even on $\C_+^n$, $n\geq 3$,
 see Remark \ref{remark-estimate}).
So we study the projective metric (introduced in \cite{Rugh07}) associated to $d_C$, namely
$\td_C(x,y)=\inf\sum d_C(x_i,x_{i+1})$. The preceding inequalities show that on our family of cones, $\td_C$
is nondegenerate, and controls $d_C$. So, being of finite diameter for $\td_C$ implies being of finite
diameter for $d_C$. Therefore, $\td_C$ also obey a contraction principle. Even though $\td_C$ satisfies the
triangular inequality and also a contraction principle, it should be noted that it is more difficult to
estimate $\td_C$ than $d_C$ (and so, more difficult than $\delta_C$).

Acknowledgments: The author expresses his deep thanks to H.-H. Rugh for his
constant support and for stimulating discussions during the preparation of this work.

\section{A contraction principle} \label{sect-CP}

If $V_{\R}$ is a real Banach space, and $C_{\R}\subset V_{\R}$ a proper closed convex cone, the
Hilbert metric $h_{C_{\R}}$ of the cone $C_{\R}$ may be defined for $x$, $y\in C_{\R}\MZ$ by (see
\cite{Bir57} and
\cite{Bir67})
\begin{align} \label{eq-hilbert}
 h_{C_{\R}}(x,y)&=\log\frac{b}{a}, & b&=\inf\{t>0:tx-y\in C_{\R}\}, \nonumber\\ & & a&=\sup\{t>0:tx-y\in
(-C_{\R})\}.
\end{align}
On the other hand, the convexity of $C_{\R}$ permits to define a dual cone, and the Hilbert metric can be
recovered from
\begin{equation} \label{eq-dual-hilbert}
 h_{C_{\R}}(x,y)= \sup\left\{\log\frac{\langle f,y\rangle\langle g,x\rangle}
{\langle f,x\rangle\langle g,y \rangle}\right\},
\end{equation} where the supremum is taken over all $f$, $g\in V_{\R}'$, nonnegative on $C_{\R}$ and such
that
$\langle f,x\rangle$, $\langle g,y\rangle >0$.

Let $V$ be a complex Banach space, $V'$ its dual and $\langle\cdot,\cdot\rangle$ the canonical duality
$V'\times V\to \C$.

\begin{definition}
Let $C\subset V$ be a non-empty subset. $C$ is said to be 
\begin{itemize}
\item[-] a {\em complex cone} if $\C^* C\subset C$.
\item[-] {\em proper} if the closure $\overline{C}$ of $C$ contains no complex planes.
\end{itemize}
\end{definition}
Note that we make no topological assumption on the cone here. When the cone $C$ is closed, this definition
of properness is the one given in \cite{Rugh07}.
\begin{definition}
Let $C$ be a proper complex cone. Let $x,\,y\in C\MZ$. We consider the set $E(x,y)$ defined by
$$ E(x,y)=E_C(x,y)=\{z\in \C:\:zx-y\notin C\}. $$
We then define:
\begin{itemize}
\item[-] $\delta_C(x,y) = 0$ if $x$ and $y$ are colinear.
\item[-] $\delta_C(x,y) = \log\dfrac{b}{a}\in[0,\infty]$ if $x$ and $y$ are linearly independent, where
\begin{align*}
 b&=\sup |E(x,y)|\in (0,\infty] & \text{ and }& & a&=\inf|E(x,y)| \in [0,\infty).
\end{align*}
Note that by properness of $C$, $E(x,y)\neq\emptyset$, and that we always have $0\notin E_C(x,y)$.
\end{itemize}
\end{definition}
We always have $\delta_C(x,y)=\delta_C(y,x)$, and $\delta_C(x,\lambda y)=\delta_C(x,y)$,
$\lambda\in\C^*$. If $\delta_C(x,y)=0$ then $x$ and $y$ must be colinear. Indeed, suppose they are
independent. The set $E_C(x,y)$ is included in a circle $C(0,M)$. So the vector plane spanned by $x$ and $y$
is included
in the closure of $C$, and this is impossible since we assume the cone proper. However, $\delta_C$ need not
satisfy the triangular inequality on a {\em general} complex cone.
\begin{definition}
Following \cite{An04} (see also \cite{Hor94}), we say that a complex cone $C\subset V$ is {\em linearly 
convex} if through each point in the complement of $C$, there passes a complex hyperplane not intersecting
$C$. If a complex cone $C$ is linearly convex, one defines its {\em dual complement} to be the set
\begin{equation} \label{eq-dual}
 C'=\{f\in V':\:\forall x\in C,\,\langle f,x\rangle \neq 0\}.
\end{equation}
\end{definition}
Note that if $C$ is linearly convex and $C\neq V$, then $0\notin C$, $0\notin C'$. We have also the following
characterization.
For all $x\in V$,
\begin{equation} \label{eq-char-lc}
 x\in C \quad\iff\quad \forall f\in C',\,\langle f,x\rangle \neq 0.
\end{equation}
\begin{lemma} \label{lemma-delta}
Let $C$ a linearly convex proper complex cone, then $\delta_C$ is a {\em projective metric} on $C$,
and we have the formula
\begin{equation} \label{eq-delta_C-lc}
 \delta_C(x,y)=\sup_{f,\,g\in C'} \log \left| \frac{\langle f,y\rangle \langle g,x\rangle}{\langle
f,x\rangle\langle g,y\rangle} \right|.
\end{equation}
\end{lemma}

\begin{proof}
Let $C$ be a linearly convex proper complex cone,
and $x,\,y\in C$. Then, by (\ref{eq-char-lc}), $zx-y\notin C$ iff $\langle f, zx-y\rangle =0$ for some
$f\in C'$. So we have
$$ E_C(x,y)=\left\{ \frac{\langle f,y\rangle}{\langle f,x\rangle} :\: f\in C'\right\}. $$ This gives
(\ref{eq-delta_C-lc}) and the triangular inequality then follows from (\ref{eq-delta_C-lc}).
\end{proof}

To ensure completeness of our metric, we need a regularity condition on the cone.
\begin{definition}
Let $K\geq 1$. Following \cite{Rugh07}, we say that a complex cone $C$ is of $K$-bounded sectional aperture if for
each vector subspace $P$ of (complex) dimension 2, one may find $m=m_P \in V'$, $m\neq 0$ such that
\begin{equation} \label{eq-KBSA}
\forall u\in C\cap P,\quad 
\|m\|\cdot\|u\| \leq K |\langle m,u\rangle|.
\end{equation}
\end{definition}
In finite dimension, a proper complex cone $C$ is automatically of bounded sectional aperture for some $K\geq 1$
(see \cite{Rugh07}, Lemma 8.2).

\begin{lemma} \label{lemma-completeness}
Let $C$ a proper complex cone of $K$-bounded sectional aperture.
\begin{enumerate}
\item \label{lemma-comp1}
For any $x$, $y\in C$ with $\|x\|=\|y\|=1$, there exists $\alpha\in\C$, $|\alpha|=1$ such that
$ \|\alpha y - x\|\leq K \delta_C(x,y).$
\item Suppose in addition that $C$ is linearly convex.
Then $(C/\sim,\delta_C)$ is a complete metric space, where $x \sim y$ iff $\C^* x = \C^* y$.
\end{enumerate}
\end{lemma}

\begin{proof}
Assume $x$ and $y$ independent and $\delta_C(x,y)<\infty$. Take $m\in V'$, $\|m\|=1$
satisfying (\ref{eq-KBSA}) for the vector
plane spanned by $x$ and $y$. Define $x'=x/\langle m,x\rangle$ and $y'=y/\langle m,y\rangle$.
One has $1\in E_{C}(x',y')$, since otherwise $0=K |\langle m, x'-y' \rangle |\geq \|x'-y'\|$.
Let $0<a<\inf|E_{C}(x',y')|$ and $b>\sup|E_{C}(x',y')|$. Then $b$, $a^{-1}>1$. Using (\ref{eq-KBSA}),
\begin{eqnarray*}
\| x'-y'\| &\leq&  \frac{b-1}{b-a} \|ax'-y'\| + \frac{1-a}{b-a}\|bx'-y'\|
\leq 2K\frac{(b-1)(1-a)}{b-a} \\ 
&\leq& 2K\frac{\sqrt{b} -\sqrt{a}}{\sqrt{b} +\sqrt{a}}
 = 2K\tanh \frac{\log(b/a)}{4} \leq (K/2)\log(b/a). 
\end{eqnarray*}
Finally, for some $\alpha\in\C$, $|\alpha| =1$,
$$\|\alpha y -x\| = \left\| \dfrac{x'}{\|x'\|} - \dfrac{y'}{\|y'\|} \right\|
\leq 2\dfrac{\|x'-y'\|}{\max(\|x'\|,\|y'\|)} \leq 2\|x'-y'\|.$$

Now, let $C$ be linearly convex, and let $(x_n)$ be a Cauchy sequence for $\delta_C$, with $\|x_n\|=1$. We
choose a subsequence $(y_n)$ such that $\delta_C(y_n,y_{n+1})<2^{-n}$. By \ref{lemma-comp1}, we may rotate
inductively $y_n$ in order to have $\|y_{n+1}-y_n\|\leq K2^{-n}$. Thus, $y_n \to y$, $\|y\|=1$,
 since $V$ is a Banach space.
Let us show that $y\in C$.
If $\langle g,y\rangle =0$ for all $g\in C'$ then for $x\in C$ and $\lambda\in\C$, one has
$\langle g,x+\lambda y\rangle = \langle g,x\rangle \neq 0$, for all $g\in C'$, hence $x+\lambda y\in C$. This is
impossible by properness of $C$, so we can choose $g_0\in C'$ such that $\langle g_0,y\rangle \neq 0$. Let $f\in C'$.
From (\ref{eq-delta_C-lc}), we have 
$$ \left|\frac{\langle g_0,y_p \rangle}{\langle g_0, y_n \rangle}\right|
 \leq \exp(\delta_C(y_p,y_n)) \left|\frac{\langle f,y_p \rangle}{\langle f, y_n \rangle}\right|. $$
Letting $p\to \infty$, we see that $\langle f,y\rangle \neq 0$, for all $f\in C'$ so $y\in C$. Finally, $\delta_C$
is lower semi-continuous by (\ref{eq-delta_C-lc}). So $\delta_C(y_n,y)\leq \liminf_{p\to\infty} \delta_C(y_n,y_p)$,
and $\delta_C(y_n,y)\to 0$.
\end{proof}

We now come to the contraction principle.
\begin{theorem} \label{th-contraction-delta}
Let $V_1,\,V_2$ be complex Banach spaces, and let $C_1\subset V_1$, $C_2\subset V_2$ be proper complex
cones (not necessarily linearly convex). Let $T:V_1 \to V_2$ be a linear map, and suppose that $T(C_1\MZ)\subset
C_2\MZ$. If the diameter $\Delta=\sup_{x,y\in C_1\MZ} \delta_{C_2} (Tx,Ty)$ is finite, then we have
$$ \forall x,\,y\in C_1,\quad \delta_{C_2}(Tx,Ty) \leq \tanh\left(\frac{\Delta}{4}\right) \delta_{C_1}
(x,y).$$
\end{theorem}

\begin{proof}
Let $x,\,y\in C_1\MZ$. We may assume $Tx$ and $Ty$ linearly independent, and $\delta_{C_1}(x,y)<\infty$. If
$z\in \C$, and $zx-y\in C_1$ then we have also $zTx-Ty \in C_2$. So $E_{C_2}(Tx,Ty)\subset E_{C_1}(x,y)$.
Let $\lambda,\,\mu$ belong to the complement of $E_{C_1}(x,y)$, $\lambda\neq\mu$, and $\alpha,\,\beta\in
E_{C_2}(Tx,Ty)$. 
So we have
\begin{align} \label{eq-def-ablm}
\lambda x - y,\: \mu x -y &\in C_1, & \alpha Tx-Ty,\:\beta Tx-Ty &\notin C_2.
\end{align}
If $z\in\C$, $z\neq 1$, then $z(\lambda Tx-Ty) - (\mu Tx-Ty)\in C_2$ if and only if
$h(z)Tx-Ty\in C_2$, where $h(z)=(z\lambda-\mu)/(z-1)$ is a Möbius transformation. We deduce that $h^{-1}
E_{C_2}(Tx,Ty) = E_{C_2}(T(\lambda x-y),T(\mu x-y))$. Thus, $(\mu-\alpha)/(\lambda-\alpha)$ and
$(\mu-\beta)/(\lambda-\beta)$ both belong to $E_{C_2}(T(\lambda x -y), T(\mu x -y))$. Since $\lambda x -
y$, $\mu x-y\in C_1\MZ$, we have $\delta_{C_2}(T(\lambda x -y), T(\mu x -y))\leq \Delta$. Therefore, we
have proved that for arbitrary $\alpha$, $\beta$, $\lambda\neq \mu$ satisfying (\ref{eq-def-ablm})
\begin{equation} \label{eq-ct-1}
\left|\frac{\mu-\alpha}{\lambda-\alpha}\cdot\frac{\lambda-\beta}{\mu-\beta} \right| \leq e^{\Delta}.
\end{equation}

We consider now $M>\sup |E_{C_1}(x,y)|$, and $m<\inf |E_{C_1}(x,y)|$, $m>0$. We fix $\alpha,\,\beta\in
E_{C_2}(Tx,Ty)$, we define $A=|\alpha|$, $B=|\beta|$. We assume also that $A<B$. Since
$E_{C_2}(Tx,Ty)\subset E_{C_1}(x,y)$, we have $m<A<B<M$. Moreover, the whole circles $C(0,m)$ and $C(0,M)$
of center $0$ and radius $m$ and $M$ respectively are included in the complement of $E_{C_1}(x,y)$. So we
may choose $\mu\in C(0,M)$, $\lambda\in C(0,m)$ to optimize the inequality (\ref{eq-ct-1}). More precisely,
we have:
\begin{itemize}
\item[-] The Möbius transformation $z\mapsto \dfrac{z-\beta}{z-\alpha}$ maps the circle $C(0,m)$ onto the
circle of center $\dfrac{\beta\overline{\alpha} -
m^2}{A^2-m^2}$ and radius $\dfrac{m|\beta-\alpha|}{A^2-m^2}$.
\item[-] The Möbius transformation $z\mapsto \dfrac{z-\alpha}{z-\beta}$ maps the circle $C(0,M)$ onto the
circle of center $\dfrac{M^2-\overline{\beta}\alpha}{M^2-B^2}$ and
radius $\dfrac{M|\beta-\alpha|}{M^2-B^2}$.
\end{itemize}
Since the element of greatest modulus in the circle of center $c$ and radius $r$ is of modulus $|c| +r$,
(\ref{eq-ct-1}) gives
$$ \frac{ \Big(|\beta\overline{\alpha} - m^2| + m|\beta - \alpha|\Big)
\Big( |M^2-\overline{\beta}\alpha| + M|\beta-\alpha|\Big)}{(A^2-m^2)(M^2-B^2)} \leq e^{\Delta}. $$
Now we have $|\beta\overline{\alpha} - m^2| + m|\beta - \alpha|\geq AB-m^2 +m(B-A)= (A+m)(B-m)$. In the
same way, we find $|M^2-\overline{\beta}\alpha| + M|\beta-\alpha|\geq (M-A)(M+B)$. So we get
\begin{equation} \label{eq-ct-2}
\frac{(M-A)(B-m)}{(M-B)(A-m)} \leq e^{\Delta}.
\end{equation} At this point, we are back to the case of the Hilbert
metric. We write $d=\log(B/A)>0$, $D=\log(M/m)>d$. Consider $\phi(t)=(M-t)(e^d t-m)(M-e^d
t)^{-1}(t-m)^{-1}$, with $t\in(m,Me^{-d})$. Then $\phi$ has a minimum at $t_0=\sqrt{Mm}e^{-d/2}
\in(m,Me^{-d})$. So (\ref{eq-ct-2}) gives $$ \phi(t_0)=
\frac{\sinh^2\left(\frac{D+d}{4}\right)}{\sinh^2\left(\frac{D-d}{4}\right)}\leq \phi(A) \leq e^{\Delta}.$$
This leads to $\tanh(d/4)\leq \tanh(\Delta/4)\tanh(D/4)$, so $d\leq \tanh(\Delta/4) D$ and the conclusion
follows.
\end{proof}

\begin{remark} We proved exactly $\tanh(\delta(Tx,Ty)/4)\leq \tanh(\Delta/4)\tanh(\delta(x,y)/4)$. So, if we
define the {\em contraction coefficient} of $T$ to be $c(T)=\tanh(\Delta/4)$, then we have $c(TS)\leq
c(T)c(S)$.
\end{remark}

\begin{example} \label{example1}
Let $C$ be a closed proper complex cone with non-empty interior, and let also $A$ be a linear map such that
$A(C\MZ)\subset \inte\,C$. Suppose that there exists $0<\rho<1$, such that $B(Ax,\rho \|Ax\|) \subset
\inte\,C$ for all $x\in C\MZ$ (such a $\rho>0$ always exists in finite dimension). Then the
$\delta_C$-diameter of $AC$ is bounded by $\Delta=2\log (1/\rho)$, hence finite. If we assume in addition that
$C$ is of $K$-bounded sectional aperture for some $K\geq 1$ (which is also automatic in finite dimension
for a proper cone, see \cite{Rugh07}), then $A$ has a spectral gap. This is proved exactly in the same way
as Theorem 3.6 and 3.7 of \cite{Rugh07}, replacing the gauge $d_C$ by $\delta_C$, the contraction principle
by Theorem \ref{th-contraction-delta} and using Lemma \ref{lemma-completeness}, 1. Moreover,
the size of the spectral gap is given by $\tanh(\Delta/4)=(1-\rho)/(1+\rho)$. We dot not reproduce Rugh's proof.

Suppose now that $V$ is finite dimensional. Then the fact that that $A$ has a spectral
gap is the content of Theorem 8.4 of \cite{Rugh07}.
However, in this situation, the sets $\RS\setminus E_C(Ax,Ay)$ might not even be connected, and
so one cannot say anything on the $d_C$-diameter of $AC$. The technical argument of \cite{Rugh07} 
does not lead to an explicit and easy estimate of the spectral gap and is simplified
by the use of $\delta_C$.
\end{example}

\section{The canonical complexification $\C_+^n$ of $\R_+^n$} \label{sect-CPN}

We will denote respectively by $\Re(z)$ and $\Im(z)$ the real and imaginary part of $z\in\C$.
Following \cite{Rugh07}, we define
\begin{equation} \label{eq-defcpn}
\C_+^n=\{v\in\C^n:\: \forall k,l,\, \Re(v_k\overline{v_l})\geq 0\}.
\end{equation}
$\C_+^n$ is a closed complex cone which is obtained from the Birkhoff cone $\R_+^n$ in a natural way as
described in \cite{Rugh07}. The interior of $\C_+^n$ is given by
\begin{equation} \label{eq-intcpn}
\inte\,\C_+^n =\{v\in\C^n:\: \forall k,l,\, \Re(v_k\overline{v_l})> 0\}.
\end{equation}
We consider on $\C^n$ the duality $\langle x,y\rangle=\sum x_k y_k$.
Our study of $\C_+^n$ is based on the following two lemmas.
\begin{lemma} \label{LC-CPN}
$\C_+^n\MZ$ and $\inte\,\C_+^n$ are both linearly convex. More precisely, we have
\begin{eqnarray} \label{eq-cpn}
x\in \inte\,\C_+^n &\iff& \forall y\in\C_+^n\MZ,\:\:\langle y,x\rangle \neq 0, \nonumber \\
y\in \C_+^n\MZ &\iff & \forall x\in\inte\,\C_+^n,\:\:\langle y,x\rangle \neq 0.
\end{eqnarray}
\end{lemma}

\begin{proof}
Let $x\in\C_+^n\MZ$ and $y\in\inte\,\C_+^n$. Then, up to multiplying $x$ and $y$ by nonzero complex
numbers, we may write $x_k=r_k e^{i\alpha_k}$, $y_k=s_k e^{-i\beta_k}$ where $r_k\geq 0$, $s_k>0$,
$\alpha_k\in [0,\pi/2]$, $\beta_k\in (0,\pi/2)$. Thus, 
$\Re\langle x,y\rangle = \sum r_k s_k\cos(\alpha_k-\beta_k) >0$, hence $\langle x,y\rangle \neq 0$.

Let $x\in \C^n$ such that $\Re(x_k\overline{x_l})<0$ for some $k\neq l$. We define
$a=\sum_{j\neq k,l} x_j$. Then for $\epsilon>0$ sufficiently small, $\Re((\epsilon a + x_k)\overline{x_l})<0$.
We write $\epsilon a +x_k=re^{i\alpha}$, $x_l=se^{i\beta}$ with $r,s>0$ and $\mu= \pi -\beta +\alpha 
\in (-\pi/2,\pi/2)$. We then define $y_k=1$, $y_l=rs^{-1} e^{i\mu}$ and $y_j=\epsilon$ for $j\neq k,l$.
Then $y\in\inte\,\C_+^n$ and $\langle x,y\rangle=0$. This proves the second part of (\ref{eq-cpn}).
The first part is proved essentially in the same way.
\end{proof}

\begin{lemma} \label{DF-CPN}
Let $x$, $y\in \inte\,\C_+^n$. Then $$E_{\inte\C_+^n}=\bigcup_{k,l} \overline{D}_{kl}(x,y)$$ where
 $\overline{D}_{kl}=\overline{D}_{kl}(x,y)$ is the closed disk of center $c_{kl}(x,y)$ and radius
$r_{kl}(x,y)$. $c_{kl}$ and $r_{kl}$ are 
given 
by
\begin{equation}  \label{eq-ckl-rkl}
c_{kl}(x,y)= \frac{\overline{x_l}y_k + \overline{x_k}y_l}{2\Re(x_k\overline{x_l})},\quad
r_{kl}(x,y)= \frac{ |x_l y_k - x_k y_l| }{2\Re(x_k\overline{x_l})}.
\end{equation}
\end{lemma}

\begin{proof}
Let $z\in\C$. Then, $z\in E_{\inte\,\C_+^n}(x,y)$ if and only if there exist $k$ and $l$ such that
\begin{equation} \label{eq-temp1}
 \Re\big((zx_l-y_l)\overline{(zx_k-y_k)}\big) \leq 0.
\end{equation} If $k$ and $l$ are such that $x_l y_k - x_k y_l=0$, then $\overline{D}_{kl}$ reduces
to $\{y_k/x_k\}$ and the only $z$ satisfying (\ref{eq-temp1}) is $y_k/x_k$. If $k$ and $l$ are such
that $x_l y_k -x_k y_l\neq 0$ then
$$ \Re\big((zx_l-y_l)\overline{(zx_k-y_k)}\big) \leq 0 \quad \iff \quad \varphi_{kl}^{-1}(z)
\in P=\{ w:\Re(w)\geq 0\} \cup \{\infty\}, $$
where $z=\varphi_{kl}(w)$ is the Möbius tranformation defined by
$$ \varphi_{kl} (w) = \frac{wy_k+y_l}{wx_k+x_l}, \quad \varphi_{kl}^{-1}(z)
=\frac{zx_l -y_l}{-zx_k +y_k}. $$
Thus the complex numbers $z$ satisfying (\ref{eq-temp1}) are exactly the elements of
$\overline{D}_{kl}=\varphi_{kl}(P)$. A brief computation then leads to the given formulas for $c_{kl}$
and $r_{kl}$.
\end{proof}

We study now the $n\times n$ matrices $A$ satisfying
\begin{equation} \label{eq-matA}
A(\C_+^n \MZ)\subset \inte\,\C_+^n.
\end{equation}
By Lemma \ref{LC-CPN}, (\ref{eq-matA}) is satisfied if and only if $\langle Ax,y\rangle \neq 0$ for all $x$,
$y\in\C_+^n\MZ$. So the set of all matrices $A$ satisfying (\ref{eq-matA}) is itself a linearly convex
complex cone. Moreover, Lemma \ref{LC-CPN} also implies that if $A$ satisfies (\ref{eq-matA}) then so does its
transpose matrix ${}^t\!A$.

\begin{proposition}
Let $A=(a_{ij})_{1\leq i,j\leq n}$ be a $n\times n$ complex matrix. Then $A(\C_+^n\MZ)\subset
\inte\,\C_+^n$ if and only if for all indices $k,l,p,q$
\begin{equation} \label{eq-cond-matA}
 \Re(\overline{a_{kp}} a_{lq} + \overline{a_{kq}} a_{lp} ) > | a_{kp} a_{lq} - a_{kq} a_{lp} |. 
\end{equation}
\end{proposition}

\begin{proof}
Let us denote by $\lambda_1,\cdots,\lambda_n$ the lines of $A$. Suppose that (\ref{eq-matA}) holds. Then
for $x\in\C_+^n\MZ$, we have
$$ Ax = (\langle \lambda_1, x\rangle, \cdots, \langle \lambda_n,x\rangle ) \in \inte\,\C_+^n.$$
So, $\langle \lambda_j,x\rangle \neq 0$ for all $x\in \C_+^n\MZ$ and Lemma \ref{LC-CPN} implies that
$\lambda_j\in\inte\,\C_+^n$.

Consider now a matrix $A$ such that $\lambda_j \in \inte\,\C_+^n$ for all indices $j$. Then (\ref{eq-matA})
holds if and only if for all indices $k$, $l$ and all $x\in\C_+^n\MZ$ we have
$ \Re\big(\overline{\langle \lambda_k,x\rangle } \langle \lambda_l,x\rangle \big)> 0 $ or equivalently
$\Re({\langle \lambda_l,x\rangle}/{\langle \lambda_k,x\rangle } ) >0.$
By Lemma \ref{LC-CPN} (see also the proof of Lemma \ref{lemma-delta}), for fixed $k$ and $l$, the set of all
$\langle
\lambda_l,x\rangle / \langle \lambda_k,x\rangle$ is exactly $E_{\inte\,\C_+^n}(\lambda_k,\lambda_l)$. So
(\ref{eq-matA}) holds if and only if for all indices $k$, $l$ we have
$$E_{\inte\,\C_+^n}(\lambda_k,\lambda_l)\subset\{z\in\C:\:\Re(z)>0\}.$$ By Proposition \ref{DF-CPN}, we have
$E_{\inte\,\C_+^n}(\lambda_k,\lambda_l)=\bigcup
\overline{D}_{pq}(\lambda_k,\lambda_l)$. $\overline{D}_{pq}(\lambda_k,\lambda_l)$ is the closed disk of
center $c_{pq}(\lambda_k,\lambda_l)$ and radius $r_{pq}(\lambda_k,\lambda_l)$ where
\begin{equation} \label{eq-cpq-rpq}
 c_{pq}(\lambda_k,\lambda_l)=\frac{\overline{a_{kp}} a_{lq} + \overline{a_{kq}}
a_{lp}}{2\Re(a_{kp}\overline{a_{kq}})}, \quad r_{pq}(\lambda_k,\lambda_l)=\frac{|a_{kp}a_{lq}
-a_{kq}a_{lp}|}{2\Re(a_{kp}\overline{a_{kq}})}.
\end{equation}
Now, (\ref{eq-matA}) holds if and only if for all indices $k$, $l$, $p$, $q$, one has
$$ \Re(c_{pq}(\lambda_k,\lambda_l))>r_{pq}(\lambda_k,\lambda_l), $$ which provides the desired formula.
Finally, observe that if (\ref{eq-cond-matA}) is satisfied for all indices, then letting $k=l=j$, we have
$\lambda_j\in\inte\,\C_+^n$.
\end{proof}

\begin{corollary}
If a complex $n\times n$ matrix $A$ satisfies (\ref{eq-cond-matA}) then $A$ has a spectral gap.
\end{corollary}

\begin{proof}
This is a direct consequence of Theorem 8.4 of \cite{Rugh07}. See also Example \ref{example1}.
\end{proof}

We give now explicit estimates for the $\delta$-diameter of $A(\C_+^n\MZ)$.
\begin{proposition}
Let $A$ be a complex $n\times n$ matrix such that $A(\C_+^n\MZ)\subset \inte\,\C_+^n$. Denote by
$\lambda_1$,...,$\lambda_n$ the lines of $A$. Define
$$ \Delta_1 = \underset{k,l}{\max}\:\delta(\lambda_k,\lambda_l),\quad \Delta_2= \underset{k,l}{\max}\:
{\diam}_{\textrm{RHP}}\: E_{\inte\,\C_+^n}(\lambda_k,\lambda_l),$$ where
$\diam_{\textrm{RHP}}$ denotes the diameter with respect to the Poincaré metric of the right half
plane. Then the $\delta$-diameter $\delta\dash\diam\,\big(A(\C_+^n\MZ)\big)$ satisfies
$$ \max(\Delta_1,\Delta_2) \leq \delta\dash\diam\,\big(A(\C_+^n\MZ)\big) \leq \Delta_1 +2\Delta_2. $$
\end{proposition}

\begin{proof}
We denote by $\Delta$ the $\delta$-diameter of $A(\C_+^n\MZ)$. We denote also by $\rho(a,b)$ the Poincaré
metric in the right half plane: for $a$, $b$ with $\Re(a)$, $\Re(b)>0$
\begin{equation} \label{eq-rho-ab}
\rho(a,b)=\log\frac{|a+\overline{b}| + |a-b|}{|a+\overline{b}| -|a-b|}\geq \log\frac{\Re(b)}{\Re(a)}.
\end{equation}
Consider two vectors $u$ and
$v\in\inte\,\C_+^n$. Then from the description given by Proposition \ref{DF-CPN}, one has
\begin{eqnarray}
\delta(u,v)&=& \log\frac{\underset{k,l}{\max}\:\left\{\big(|\overline{u_k}v_l +\overline{u_l}v_k| + |u_k
v_l -u_l v_k|\big)\big(2\Re(\overline{u_k} u_l)\big)^{-1}\right\}}
{\underset{k,l}{\min}\:\left\{\big(|\overline{u_k}v_l +\overline{u_l}v_k| - |u_k v_l -u_l
v_k|\big)\big(2\Re(\overline{u_k} u_l)\big)^{-1}\right\}} \label{eq-delta-uv}\\
&=&
\underset{k,l,p,q}{\max}\: \log \frac{|u_p v_k|}{|u_k v_p|}
\frac{\Re\left(\dfrac{u_q}{u_p}\right)}{\Re\left(\dfrac{u_l}{u_k}\right)} \frac{\left|\dfrac{v_l}{v_k} +
\dfrac{\overline{u_l}}{\overline{u_k}} \right| + \left|\dfrac{v_l}{v_k} - \dfrac{u_l}{u_k}
\right|}{\left|\dfrac{v_q}{v_p} +
\dfrac{\overline{u_q}}{\overline{u_p}} \right| - \left|\dfrac{v_q}{v_p} - \dfrac{u_q}{u_p} \right|}
\label{eq-deltaxy-1} \\
&\leq&
\underset{k,l,p,q}{\max}\:\Big\{
\rho\left(\frac{v_l}{v_k},\frac{u_l}{u_k}\right) +
\rho\left(\frac{v_q}{v_p},\frac{u_q}{u_p}\right) 
+\log\left|\frac{v_k u_p}{u_k v_p}\right|\Big\}. \label{eq-deltaxy-2}
\end{eqnarray}
We used in (\ref{eq-deltaxy-2}) inequality (\ref{eq-rho-ab}) and the following straightforward identity
$$\frac{\Re(d)}{\Re(b)} \frac{|a+\overline{b}| + |a-b|}{|c+\overline{d}|-|c-d|} =
\sqrt{\frac{\Re(a)\Re(d)}{\Re(b)\Re(c)}}\exp\left(\frac{1}{2}\rho(a,b)+\frac{1}{2}\rho(c,d)\right). $$

Let $x$ $y\in\C_+^n\MZ$. We define $u=Ax$ and $v=Ay$. Thus $u_j=\langle \lambda_j,x\rangle$ and
$v_j=\langle \lambda_j,y\rangle$. The ratios $u_l/u_k$ and $v_l/v_k$ belong to
$E_{\inte\,\C_+^n}(\lambda_k,\lambda_l)$. Therefore, letting $p=k$, $q=l$ in (\ref{eq-deltaxy-1}) we get
$\Delta\geq\Delta_2$. We have obviously $\Delta\geq \Delta_1$ so the lower bound follows.
The upper bound is a consequence of inequality (\ref{eq-deltaxy-2}) and the formula for
$\delta(\lambda_k,\lambda_p)$ given by Lemma \ref{lemma-delta}.
\end{proof}

\begin{remark}
The formula (\ref{eq-delta-uv}) given in the preceding proposition shows that the projective metric
$\delta_{\C_+^n}$ extends the Hilbert metric $h_{\R_+^n}$ on $\R_+^n$. This is also the case for the
hyperbolic gauge of Rugh. So, $\delta_{\C_+^n}$ is another possible extension of the Hilbert metric.
\end{remark}

\begin{theorem} \label{th-estimate}
Let $\theta\in (0,1)$, and $\sigma>1$. Consider a complex matrix $A$ such that for all indices
\begin{eqnarray*}
 \Re(\overline{a_{kp}} a_{lq} + \overline{a_{kq}} a_{lp} ) &>&\frac{1}{\theta} | a_{kp} a_{lq} - a_{kq}
a_{lp} |, \\
|a_{kp}a_{lq}|&\leq& \sigma^2 |a_{kq}a_{lp}|.
\end{eqnarray*}
Then $A(\C_+^n\MZ)\subset\inte\,C_+^n$, and we have
$$ \delta\dash\diam\,A(\C_+^n\MZ) \leq 8 \log \frac{1+\theta}{1-\theta} + 2\log \sigma. $$
\end{theorem}

\begin{proof}
We use the same notations as in the proceding propositions. We want first to estimate $\Delta_2$. We have
$$E_{\inte\,C_+^n}(\lambda_k,\lambda_l) = \underset{p,q}{\bigcup} \overline{D}_{pq}(\lambda_k,\lambda_l). $$
Moreover, the complex numbers $a_{lp}/a_{kp}$ and $a_{lq}/a_{kq}$ both belong to $\overline{D}_{pq}$. So the
disk $\overline{D}_{pq}$ intersects the disk $\overline{D}_{qr}$ which in turn intersects
$\overline{D}_{rs}$. We deduce that
$$ \Delta_2 \leq 3 \underset{k,l,p,q}{\max}\: \diam_{\text{RPH}}\:\overline{D}_{p,q}(\lambda_k,\lambda_l).
$$
Moreover, the diameter of a closed disk $\overline{D}\subset\{\Re(z)>0\}$ of center $c$ and radius $r>0$ for
the Poincaré metric $\rho$ is $$\rho(c-r,c+r) = \log \frac{\Re(c) +r }{\Re(c)-r}. $$ So by
(\ref{eq-cpq-rpq}), $\Delta_2 \leq 3\log((1+\theta)/(1-\theta))$. Finally, we apply inequality
(\ref{eq-deltaxy-2}) to $u=\lambda_i$, $v=\lambda_j$ and we find that
$$ \Delta_1 \leq 2\log\frac{1+\theta}{1-\theta} + 2\log\sigma. $$
(One checks directly that $\rho(a_{jl}/a_{jk},\,a_{il}/a_{ik})\leq \log((1+\theta)/(1-\theta))$.)
\end{proof}

\section{More general complex cones} \label{sect-complexcones}

We will compare in section \ref{sect-comp} the metric $\delta_C$ with the
hyperbolic gauge $d_C$. As mentionned in the introduction, one cannot hope any control of $d_C$ by
$\delta_C$ in a general complex cone. Our main goal is to show (in section \ref{sect-comp}) that such a
control is possible at least for canonical 
complexification of real Birkhoff cones as defined in \cite{Rugh07}. We fix here the setting and prove the
various lemmas needed in section \ref{sect-comp}.

Recall the definition of the canonical complexification. Let $V_{\R}$ a real Banach space, and 
consider its complexification $V=V_{\R}\oplus iV_{\R}$. It is a complex Banach space. Let $C_{\R}\subset
V_{\R}$ a real Bikhoff cone and $C_{\R}'\subset V_{\R}'$ its dual. Each real linear functional on $V_{\R}$
naturally extends to a complex linear functional on $V$, and the canonical complexification is defined by
$$ C= \{x\in V:\: \forall m,l \in C_{\R}',\, \Re\big(\langle m,x\rangle \overline{\langle
l,x\rangle}\big)\geq 0\}. $$ It may also be defined as $C=\C^*(C_{\R} + iC_{\R})$.

So we naturally consider the slightly more general situation of a cone $C$ in a general complex Banach
space $V$ satisfying
\begin{equation} \label{eq-cpn-gen}
 C=\{x\in V:\:\forall m,\,l\in S,\,\Re\big(\langle
m,x\rangle\overline{\langle l,x\rangle}\big)\geq 0\},
\end{equation}
for some non-empty subset $S$ of $V'$.

\begin{lemma} \label{prop-disks}
Let $C$ satisfy (\ref{eq-cpn-gen}). Define
$$\mathcal{F}(x,y)=\{(m,l)\in S:\: \langle m,y\rangle\langle l,x\rangle -
\langle l,y\rangle \langle m,x \rangle \neq 0 \}. $$ (It may be an empty set) 
For each $(m,l)\in\mathcal{F}(x,y)$, we define also the associated open disk or open half-plane
\begin{equation} \label{eq-phiml}
 D_{m,l}(x,y) = \varphi_{ml}(\{w:\:\Re(w)>0\}), \quad
\phi_{ml}(w)= \frac{w \langle m,y \rangle + \langle
l,y\rangle}
{w \langle m,x \rangle + \langle l,x\rangle }.
\end{equation}
Then $ E(x,y) = \bigcup D_{m,l}(x,y)$, where the union is taken over all $(m,l)\in\mathcal{F}(x,y)$.
\end{lemma}

\begin{proof}
The proof is essentially the same as the proof of Lemma \ref{DF-CPN}.
\end{proof}

We will denote by $\Q$ the first quadrant of the complex plane:
$\Q=\{z\in\C:\:\Re(z)\geq 0\text{ and } \Im(z)\geq 0\}$.
In what follows, we consider a complex Banach space $V$. We consider a closed {\em convex cone} $R\subset
V$, that is, a closed non-empty subset such that $R+R\subset R$ and $(0,\infty)R\subset R$.
We consider also a complex cone $C\subset V$ satisfying
\begin{condition}[C1] $C=\C^* R$, and
$\overline{(R+iR)}\cap\overline{(R-iR)}=R.$ In particular, $C$ is closed.
\end{condition}
\begin{condition}[C2] $C$ is proper and contains at least two (complex) independent vectors.
\end{condition}

We define also $ S = \{ m \in V':\:\forall x \in R,\, \langle m,x \rangle \in \Q \}$. It is a closed
convex cone.

\begin{lemma} \label{lem-C1}
Let $C$ and $R$ satisfy (C1). Then we have
\begin{eqnarray}
 R&=&\{ x \in V:\: \forall m\in S,\, \langle m,x \rangle \in \Q\}, \label{eq-RR} \\
 C=\C^* R&= &\{x\in V:\: \forall m,l\in S,\, \Re\left(\langle m,x\rangle\overline{\langle
l,x\rangle}\right)\geq 0 \}.  \label{lem-dual}
\end{eqnarray}
\end{lemma}
\begin{proof} 
Let $x\in V$. If $x\in\overline{(R-iR)}$, then $\Re\langle m,x\rangle \geq 0$, $\forall m\in S$.
Conversely, $\overline{(R-iR)}$ is a closed convex cone. So, if $x\notin\overline{(R-iR)}$, by Mazur's
theorem, there exists $m\in V'$ such that $\Re\langle m,x\rangle <0$; and such that $\Re\langle m,y\rangle
\geq 0$, $\forall y\in\overline{(R-iR)}$, or equivalently, $m\in S$. 

Now, $\langle m,x\rangle \in \Q$ if and only if $\Im\langle m,x\rangle=\Re(\langle m,-ix\rangle)\geq 0$ and
$\Re\langle m,x\rangle \geq 0$. Thus, $x\in V$ satisfies $\langle m,x\rangle \in \Q$ for all $m\in S$ if and
only if $x$ and $-ix$ both
belong to $\overline{(R-iR)}$. Hence, by (C1), if and only if $x\in R$. This gives (\ref{eq-RR}).

Finally, let $x\in V$. Then $\Re(\langle m,x\rangle\overline{\langle l,x\rangle})\geq 0$ for all
$m,\,l\in S$ if and only if the argument of $\langle m,x\rangle$ (when $\langle m,x\rangle$ is non-zero)
varies within a $\pi/2$ angle. In other words, if and only if one may find $\alpha \in \C^*$ such that
$\langle m,\alpha x\rangle \in \Q$ for all $m\in S$. (\ref{lem-dual}) then follows from
(\ref{eq-RR}).
\end{proof}

\begin{example}
As regards the canonical complexification of a real Birkhoff cone, one takes $R=C_{\R}+iC_{\R}$, and
(C1) is a consequence of the convexity of $C_{\R}$. Another example is provided by
\begin{equation} \label{eq-ex2}
 C = \{x\in V:\:\|x-\langle \mu,x\rangle a\| \leq \sigma |\langle \mu,x\rangle|\cdot \|a\| \},
\end{equation}
where $\sigma > 0$, $a\in V$ and $\mu \in V'$ such that $\langle \mu,a\rangle=1$ ($\mu$ is a
{\em complex} linear functional). Here, one takes $R=\{x\in C:\:\langle\mu,x\rangle\in[0,\infty)\}$.
The Remark 3.10 in \cite{Rugh07} says that a bounded linear operator $T$
on a complex Banach space has a spectral gap if and only if it is a strict contraction of a cone like
(\ref{eq-ex2}) (but with another norm).
\end{example}

\begin{remark}
The condition (C2) implies that the cone $R$ satisfies $R\cap(-R)\subset \{0\}$, and so is {\em proper} 
as a convex cone. Indeed, let $x\in R\cap (-R)$. Then for all $m\in S$, $\langle m,x\rangle=0$. By
(\ref{lem-dual}), this implies that $\forall y\in C$, $\forall z,z'\in\C$, $zx+z'y \in C$.
By (C2), we find that $x$ must be $0$.
We mention also the following consequence of (C2). Given $x,y\in C$ complex linearly
independent, we can find $m\in S$ such that $\langle m,x\rangle
\neq 0$ (by properness of $R$ and(\ref{eq-RR})). But then, $y-\big(\langle m,y\rangle / \langle
m,x\rangle\big) x \neq 0$ and
we can find again $l\in S$ not vanishing on this vector. This proves that $\mathcal{F}(x,y)\neq \emptyset$.
\end{remark}

\begin{remark}
The cone $C$ need not be linearly convex (even though one can show that $\inte\,C$ is).
However, $\delta_C$
satisfies the triangular inequality on $C$. Indeed, using Lemma \ref{prop-disks}, one can show that the
formula of Lemma \ref{lemma-delta} remains valid if we replace $f\in C'$ by $f\in S+iS$, $\langle
f,x\rangle$, $\langle f,y\rangle \neq 0$. This is because each $D_{m,l}$ is an open finite disk or a half
plane and $0\notin D_{m,l}$. So the supremum and the infimum of $|z|$ are both attained on the boundary of
$D_{m,l}$. We skip the details.
\end{remark}

\begin{definition} Let $x$, $y\in C\MZ$. We define $L(x,y) = \{z \in \RS:\:zx-y \in C\}$, with the
convention that $\infty \in L(x,y)$.
\end{definition}

\begin{lemma} \label{lem-convexhull} \label{prop-SC}
Let $x,y\in C\MZ$ and suppose that $0$, $\infty$ are in the interior of $L(x,y)$ in $\RS$ (denoted by
$\mathring{L}(x,y)$). Then each $D_{m,l}$, $(m,l)\in\mathcal{F}(x,y)$ is an open finite disk. Moreover, the
closed convex hull of the centers $c_{m,l}$ is included in the closure $\overline{E(x,y)}$ of $E(x,y)$.
In general, we can say that as soon as $\mathring{L}(x,y)$ is not empty, it is a simply connected open
subset of $\RS$.
\end{lemma}

\begin{proof}
Since $\phi_{m,l}$ given
by (\ref{eq-phiml}) is a Möbius transformation, it maps the open half plane $\{\Re(w)>0\}$ onto an open
finite disk, an open half plane or the complement (in $\RS$) of a closed finite disk. Since $0$,
$\infty\in\mathring{L}(x,y)$, each $D_{m,l}$ must be an open finite disk, 
and $\phi_{m,l}^{-1}(\infty) \notin \{\Re(w)\geq0\}\cup\{\infty\}$.
So $\Re(\langle m,x\rangle\overline{\langle l,x\rangle})>0$, 
and we get the following formula for the center $c_{m,l}$.
\begin{equation} \label{lem-disks-formula}
 c_{m,l} = \frac{\overline{\langle l,x\rangle}\langle m,y\rangle + \overline{\langle m,x\rangle} \langle
l,y\rangle} {2\Re\left(\langle m,x\rangle\overline{\langle l,x\rangle}\right)}.
\end{equation}

Now, let $A\subset \mathcal{F}(x,y)$ be finite and for $(m,l)\in A$, let $t_{m,l}\geq 0$, $\sum t_{m,l}=1$.
We may assume $A$ symmetric, \ie\ $(m,l)\in A$ iff $(l,m)\in A$. Then the complex number $z=\sum_{(m,l)\in A}
t_{m,l} c_{m,l}$ satisfies $z=\langle f,y \rangle$ where $f\in V'$ is defined by
\begin{equation} \label{eq-ch1}
 \forall v\in V,\quad \langle f,v\rangle = \sum_{(m,l)\in A} \frac{(t_{m,l}+t_{l,m}) \overline{\langle
l,x\rangle}}{2 \Re\left(\langle m,x \rangle \overline{\langle l,x\rangle} \right)} \langle m,v \rangle.
\end{equation}
We observe that $\langle f,x \rangle = 1$ so that $\langle f,zx-y\rangle=0$.

Define $S_0=\{l\in V':\:\exists m\in V',\,(m,l)\in A\}$, the finite set of all the linear functionals $l$
appearing in (\ref{eq-ch1}). We suppose first that for all $l,\,l'\in S_0$, we have
$\Re\big(\langle l,x\rangle\overline{\langle l',x\rangle}\big)>0$. Then, we can find $\beta\in\C^*$ such
that for all $l\in S_0$, $\beta\langle l,x\rangle $ belongs to the interior of $\Q$. Let $u\in R$, and
suppose that $\langle f,u\rangle =0$. Taking real parts in (\ref{eq-ch1}) with $v=\overline{\beta}u$, we
see that as soon as $t_{m,l} +t_{l,m}>0$, we have $\langle m,u\rangle =0$, and by symmetry, $\langle
l,u\rangle=0$. 
Now, we observe that for all $(m,l)\in\mathcal{F}$, one cannot have simultaneously
$\langle m,zx-y\rangle =0$ and $\langle l,zx-y\rangle =0$. Since $\langle f,zx-y\rangle =0$, we deduce that
$zx-y\notin C=\C^*R$. This proves that $z\in E(x,y)$.

We consider now the general case. We fix $\mu\in S$ such that $\langle \mu,x\rangle \neq 0$ (possible by
properness of $R$). If $\epsilon>0$ and $m\in S$, we define $m_{\epsilon}=m + \epsilon \mu\in S$. A direct
calculation then shows that for all $l,l'\in S$, $\Re(\langle l_{\epsilon}, x \rangle \overline{\langle
{l'}_{\epsilon}, x \rangle} ) \geq \epsilon^2 |\langle \mu,x \rangle|^2 >0$.
From the preceding discussion, we deduce that $z_{\epsilon}=\sum_{(m,l)\in A} t_{m,l}
c_{m_{\epsilon},l_{\epsilon}}$ belong to $E(x,y)$, and letting $\epsilon\to 0$, $z\in\overline{E(x,y)}$.
So the closed convex hull $K$ of the $c_{m,l}$ is a compact convex subset of
$\overline{E(x,y)}$.

Now, consider the euclidean projection $p$ on this closed convex set $K$. If $z\in
E(x,y)$, then the entire segment $[z,p(z)] \subset E(x,y)$. Indeed, $z$ belongs to some open disk
$D_{m,l}$, and since $c_{m,l}\in K$ we have the following angular condition:
$\Re\big( (z-p(z)) \cdot \overline{(c_{m,l} -p(z))} \big) \leq 0$, from which we deduce that
$|z-c_{m,l}|^2 \geq |p(z) - c_{m,l}|^2$. Hence, $p(z)$ belongs also to $D_{m,l}$ which is convex and thus
contains $[z,p(z)]$. Since $p$ is continuous, we deduce that if $z\in\overline{E(x,y)}$, then
$[z,p(z)]\subset\overline{E(x,y)}$. Therefore, if $z\in\mathring{L}(x,y)$ the
halfline $z+\R_+(z-p(z))\subset \mathring{L}(x,y)$.
It is now easy to see that every loop in $\mathring{L}(x,y)$ with base point $\infty$ is homotopic to the
constant loop at $\infty$.

In general, if $\mathring{L}(x,y)\neq\emptyset$ and if $\alpha$, $\beta\in\mathring{L}(x,y)$, $\alpha\neq
\beta$, then
the Möbius tranformation $z\mapsto (z\alpha - \beta)/(z-1)$ maps $L(\alpha x-y,\beta x-y)$ onto $L(x,y)$,
and $0$, $\infty \in \mathring{L}(\alpha x-y,\beta x-y)$. Thus, $\mathring{L}(x,y)$ is also simply
connected. 
\end{proof}

\begin{remark}
Let $\Omega\subset V$ be an open complex cone. We consider the projective space $PV=V\MZ/\sim$ associated to
$V$ (where $x\sim y$ iff $\C x=\C y$), and the open subset $P\Omega\subset PV$ associated to $\Omega$.
The open set $P\Omega\subset PV$ is said to be
{\em $\C$-convex} (see \cite{An04}) if its intersection with any projective complex line $L$ is simply
connected and $\neq L$. We refer to \cite{An04} for historical notes and more about $\C$-convexity.
It is interesting to mention that, in general, $\C$-convexity implies linear
convexity for open sets. See for instance \cite{Hor94} or \cite{An04} for a proof in finite dimension. This
result has been extended recently to complex Banach spaces in \cite{Fl06}. 
In our situation, we mention without proof that when non-empty, $\inte\,C$ is
$\C$-convex, and linearly convex.
(This is not actually a direct consequence of the preceeding lemma: one has to prove that
$ L_{\inte\,C}(x,y) = \mathring{L}_{C} (x,y)$.)
\end{remark}

\section{Comparisons of metrics} \label{sect-comp}

In this section, we consider a complex cone $C$ satisfying
conditions (C1)-(C2) of section \ref{sect-complexcones}.
In our setting, the
complex gauge $d_C(x,y)$ of Rugh (cf. \cite{Rugh07}) is defined as follows.
\begin{definition}[Rugh]
\begin{itemize}
\item[1.\ ] If $x$ and $y$ are colinear then $d_C(x,y)=0$.
\item[2.\ ] If $x$ and $y$ are (complex) linearly independent and if $\mathring{L}(x,y)$ contains $0$ and
$\infty$, then $d_C(x,y)=d_{\mathring{L}(x,y)}(0,\infty)$ is the Poincaré distance (see \eg\
\cite{Mil06}) between $0$ and $\infty$ in the hyperbolic Riemann surface $\mathring{L}(x,y)$. Note that our
case, properness of the cone implies that $\mathring{L}(x,y)$ avoids at least three points and hence is
hyperbolic. As a normalization, we consider the Poincaré metric with Gaussian curvature $-1$ (thus, on the
unit disk, it is $2|dz|/(1-|z|^2)$).
\item[3.\ ] $d_C(x,y)=\infty$ otherwise.
\end{itemize}
\end{definition}

The gauge $d_C$ does not satisfy the triangular inequality in general, even if the cone $C$ satisfy
condition (C1)-(C2). So we consider as in \cite{Rugh07} the following projective pseudo-metric.
\begin{definition}
$ \tilde{d}_C(x,y) =\inf\{ \sum d_C(x_k,x_{k+1}):\:x_0=x,x_1,\cdots,x_n=y \in C\MZ\}. $
\end{definition}

One always has $\td_C(x,y)\leq d_C(x,y)$.

By Proposition \ref{prop-SC}, when non-empty, $\mathring{L}(x,y)$ is simply connected. As
an interesting consequence, we may use the improved contraction constant in Lemma 2.4 of \cite{Rugh07}.
More precisely, we have: (see \cite{Rugh07}, Lemma 2.4 and Remark 2.5)
\begin{theorem}[Rugh] \label{th-rugh}
Let $T:V_1 \to V_2$ be a complex linear map. Let $C_1\subset V_1$,
$C_2\subset V_2$ satisfy (C1) and (C2), and such that $T(C_1\MZ)\subset C_2\MZ$. Assume that the diameter
$\Delta=\textrm{diam}_{C_2} T(C_1\MZ)$ is finite, then we have
$$ \forall x,y \in C_1\MZ,\quad d_{C_2}(Tx,Ty)\leq \tanh\left(\frac{\Delta}{2}\right)d_{C_1}(x,y).$$
\end{theorem}

Now, we compare $d_C$, $\td_C$ and $\delta_C$ on a complex cone satisfying (C1)-(C2).
First of all, let us
remark that $\delta_C$ is different from $\td_C$ (and $d_C$). For instance, in $\C_+^3$, $E(x,y)$ may look
like
figure \ref{fig-ct-ex}, where we took
\begin{equation} \label{eq-ctex}
x=(1,e^{-i\pi/12},e^{i\pi/12})\quad\textrm{ and }\quad
y=(2+e^{i\pi/3},(2-i)e^{-i\pi/12},(3-i)e^{i\pi/12}).
\end{equation}
In this case, $\delta_C(x,y)>d_C(x,y)\geq\td_C(x,y)$ since
increasing a domain decreases hyperbolic distances (observe that $\log(b/a)$ is the Poincaré distance
between $0$ and $\infty$ in $\RS\setminus D$ where $D$ is the closed disk of diameter $[a,b]$).

\begin{figure}[htb]
\begin{center}
\input{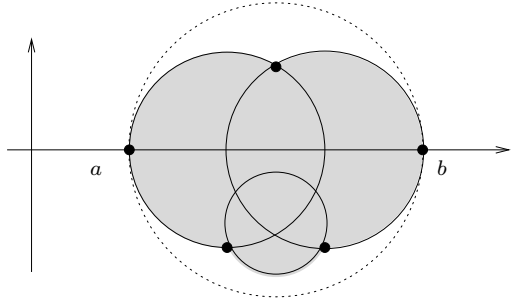}
\caption{A counter example.}
\label{fig-ct-ex}
\end{center}
\end{figure}

\begin{proposition} \label{prop-inf}
We have
$ \forall x,y\in C\MZ,\quad \frac{1}{2} \delta_C(x,y)\leq d_C(x,y)$.
\end{proposition}

\begin{proof}
We may assume that $d_C(x,y)<\infty$, and that $x$ and $y$ are linearly independent. Then $0$,
$\infty\in\mathring{L}(x,y)$. Denote by $a=\inf|E(x,y)|>0$ and $b=\sup|E(x,y)|<\infty$. Let
$z_0,\,z_1\in E(x,y)$, $z_0\neq z_1$. We can write $z_k=\phi_{m_k,l_k}(w_k)$, where $\phi_{m,l}$ is defined
by (\ref{eq-phiml}), and $(m_k,l_k)\in\mathcal{F}(x,y)$, $k=0,1$. We have $\Re(w_k)>0$ and possibly after
replacing $w_k$ by $w_k^{-1}$ and exchanging $m_k$ and $l_k$, we may assume also that $\Im(w_k)\geq 0$. We
define $f_k=w_km_k+l_k$, so that $z_k=\langle f_k,y\rangle/\langle f_k,x\rangle$, $k=0,1$. For $t\in
(0,1)$, we define also $$f_t=(1-t)f_0 + tf_1=(1-t)w_0 m_0 +
tw_1 m_1 + (1-t) l_0 + tl_1.$$ Observe that $(1-t)w_0$, $tw_1$, $(1-t)$ and $t$ all belong to
$\{w\in\C:\:\Re(w)>0\textrm{ and }\Im(w)\geq 0\}$. So, if $v\in C$ and $\langle f_t,v\rangle =0$ then
$\langle m_k,v\rangle=0$, $\langle l_k,v\rangle=0$ ($k=0,1$).
Thus, $\langle f_t,x \rangle \neq 0 $ for all $t\in [0,1]$, and we may define $z_t = \langle f_t,
y\rangle/\langle f_t,x\rangle$. Then, for all $t\in(0,1)$, $z_t\in\overline{E(x,y)}$. Indeed, if this is
not true, then $z_t x - y \in C$. But $\langle f_t,z_t x-y\rangle =0$ so we have for instance $\langle
m_0,z_tx-y\rangle = 0$ and $\langle m_0,x\rangle\neq0$. Hence $z_t=\langle m_0,y\rangle/\langle m_0,x\rangle
\in\overline{E(x,y)}$ by Proposition \ref{prop-disks}, and this is a contradiction. 

Therefore, we have a circular arc $t\mapsto z_t$, $t\in[0,1]$, with values in $\overline{E(x,y)}$. Denote by
$\Gamma$ its range. It is a compact set containing more than three points (since $z_0\neq z_1$) and not
containing $0$ and $\infty$. Increasing a domain decreases hyperbolic distances, thus
\begin{equation} \label{eq-inf1}
d_{\RS\setminus\Gamma}(0,\infty)\leq d_{\mathring{L}(x,y)}(0,\infty)=d_C(x,y).
\end{equation}
Now, the Möbius transformation
\begin{equation} \label{eq-inf2}
\varphi(z)=\frac{z\langle f_0,x\rangle - \langle f_0,y\rangle}{z\langle
f_1,x\rangle - \langle f_1,y\rangle}
\end{equation}
induces a conformal isomorphism from $\RS\setminus\Gamma$ onto
$\C\setminus\R_-$; and the map $$ \lambda\mapsto \frac{2\sqrt{\lambda}-1}{2\sqrt{\lambda}+1},\quad
\sqrt{\lambda}\textrm{ such that
}\Re(\sqrt{\lambda})>0,$$ is a conformal isomorphism from $\C\setminus\R_-$ onto the open unit disk in
$\C$. We deduce from this, that if $\lambda$, $\mu\in\C\setminus\R_-$, and $|\lambda|\geq|\mu|$, then
\begin{eqnarray*} 
\tanh\left(\frac{d_{\C\setminus\R-}(\lambda,\mu)}{2}\right)&=&\left|\frac{\sqrt{\lambda}
-\sqrt{\mu}}{ \sqrt {\lambda}+\overline{\sqrt{\mu}}}\right| \geq
\frac{|\sqrt{\lambda}|-|\sqrt{\mu}|}{|\sqrt{\lambda}|+|\sqrt{\mu}|}, \textrm{ and } \\
d_{\C\setminus\R-}(\lambda,\mu) &\geq & \log \frac{\sqrt{|\lambda|}}{\sqrt{|\mu|}} = \frac{1}{2} \log
\frac{|\lambda|}{|\mu|}.
\end{eqnarray*}
Using the Möbius transformation (\ref{eq-inf2}), we find
$$d_{\RS\setminus\Gamma}(0,\infty)=d_{\C\setminus\R_-}\left(\frac{\langle f_0,y\rangle}{\langle
f_1,y\rangle},\frac{\langle f_0,x\rangle}{\langle f_1,x\rangle}\right) \geq \frac{1}{2} \log
\left|\frac{\langle f_0,y\rangle \langle f_1,x\rangle}{\langle f_1,y\rangle\langle f_0,x\rangle} \right| =
\frac{1}{2} \log \frac{|z_0|}{|z_1|}. $$ Combining with (\ref{eq-inf1}) and the definition of $\delta_C$,
the result follows.
\end{proof}

\begin{corollary} \label{cor-tilde} Under the same assumptions, $\td_C$ does not degenerate and we have
$$\forall x,\,y\in C\MZ,\quad \frac{1}{2}\delta_C(x,y)\leq \td_C(x,y).$$
\end{corollary}

\begin{proof}
 This is a consequence of the definition of $\td_C$ and the preceding proposition.
\end{proof}

To get an upper bound for $d_C(x,y)$, we proceed in two steps. First, we consider the case when $x$ and $y$
are $\delta_C$-close to each other and then we consider the general case.

\begin{lemma} \label{lem-sup-close}
There exists a constant $\delta_0>0$
(not depending on the cone) such that if $x$, $y\in C\setminus\{0\}$ and if $$ \delta_C(x,y)<\delta_0, $$
then $$ \td_C(x,y)\leq d_C(x,y)\leq 3 \delta_C(x,y). $$
\end{lemma}

\begin{proof}
First, we consider $\alpha>1$ and $0<a<b$. Let $\Omega_{\alpha}\subset\RS$ be the complement (in $\RS$) of
the union of the two closed disks passing through $a$ and $b$ and intersecting the real axis at an angle
$\pi/(2\alpha)$ (see figure \ref{fig-omega-alpha}). Then the Poincaré distance in
$\Omega_{\alpha}$ between $0$ and $\infty$ is given by:
\begin{equation} \label{eq-capdisks}
d_{\Omega_{\alpha}}(0,\infty)= \alpha \log\frac{b}{a}.
\end{equation}
Indeed, the transformation
$$ f(z)=\left(\dfrac{z\sqrt{b}-a\sqrt{b}}{z\sqrt{a}-b\sqrt{a}}\right)^{\alpha} $$
is a conformal isomorphism from $\Omega_{\alpha}$ to the right half-plane. The Poincaré metric
on the right half-plane is given by $|dz|/\Re(z)$. Since $f(0)=f(\infty)^{-1}$ and
$f(\infty)=(b/a)^{\alpha/2}$, equation (\ref{eq-capdisks}) follows.

\begin{figure}[htb]
\begin{center}
\input{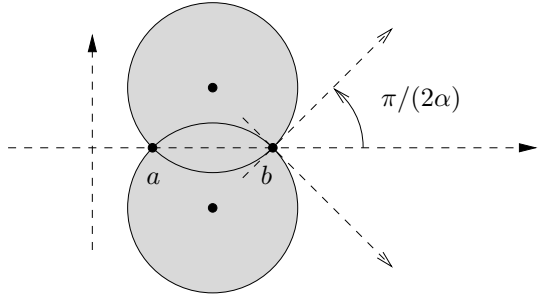}
\caption{The complement of $\Omega_{\alpha}.$}
\label{fig-omega-alpha}
\end{center}
\end{figure}

Now, let $x$, $y\in C\setminus\{0\}$. We assume that $x$ and $y$ are linearly independent and that
$\delta_C(x,y)<\infty$. Define $a=\inf|E(x,y)|>0$ and $b=\sup|E(x,y)|<\infty$. We consider a disk $D(c,r)$
of center $c$ and radius $r>0$ included in $E(x,y)$. Since each $z\in \overline{D}(c,r)$ must satisfy
$a\leq|z|\leq b$,
we have $|c|+r\leq b$ and $|c|-r\geq a$. As $0 \notin \overline{D}(c,r)$, we have for all $z\in
\overline{D}(c,r)$,
\begin{equation} \label{eq-sup-close1}
\left|\arg\frac{z}{c}\right|\leq \arcsin\frac{r}{|c|} \leq
\theta_{\max}:=\arcsin\frac{b-a}{b+a}\in(0,\frac{\pi}{2}).
\end{equation}
Let $z_1$, $z_2\in E(x,y)$. There exist $(m_1,l_1)$, $(m_2,l_2)\in\mathcal{F}(x,y)$ such that $z_k\in
D_{m_k,l_k}$. Since, $\delta_C(x,y)<\infty$, each $D_{m,l}$, $(m,l)\in\mathcal{F}(x,y)$ is an open
finite disk, whose closure passes through the two distinct points: $\langle l,y\rangle/\langle l,x\rangle$
and $\langle m,y\rangle/\langle m,x\rangle$. Since $(m_1,l_1)$, $(m_2,l_2)\in\mathcal{F}(x,y)$, one may find
a couple, say $(m_1,m_2)$, such that $\langle m_1,y\rangle/\langle m_1,x\rangle$ and $\langle
m_2,y\rangle/\langle m_2,x\rangle$ are distinct, hence $(m_1,m_2)\in\mathcal{F}(x,y)$. Thus, the disk
$\overline{D}_{l_1,m_1}$ intersects $\overline{D}_{m_1,m_2}$ which in turn intersects $\overline{D}_{m_2,l_2}$.
Since they are all included
in $\overline{E}(x,y)$, we may use inequality (\ref{eq-sup-close1}). We deduce that the argument of
$z\in E(x,y)$
varies within an angle $\leq 6\,\theta_{\max}$ provided $b/a$ is sufficiently close to $1$ (in order to have
$6\,\theta_{\max}<\pi/2$). Up to a rotation, (or equivalently, after replacing $y$ by $\lambda y$,
$\lambda\in\C$, $|\lambda|=1$), we may assume that $E(x,y)$ is included in the set
$$S(b/a):=\{z:\:\arg(z)\in [-3\theta_{\max},3\theta_{\max}],\textrm{ and }|z|\in[a,b]\}.$$

Next, we show that for $b/a$ sufficiently close to $1$ and for a suitable choice of $\alpha$ (not
depending on $b/a$) the set $S(b/a)$ is contained in $\RS\setminus\Omega_{\alpha}$. Let $r\in[a,b]$ and
$\theta\in[0,3\theta_{\max}]$. Then $re^{i\theta}\in\RS\setminus\Omega_{\alpha}$ iff
$$ \left| re^{i\theta} -\left(\frac{a+b}{2}+i\frac{b-a}{2\tan(\pi/(2\alpha))} \right)\right| \leq
\frac{b-a}{2\sin(\pi/(2\alpha))}. $$ We need only to check this inequality for $r=a$, $r=b$. In those cases,
the preceding inequality is equivalent to
$$ 1\leq \cos\theta+\frac{\sin\theta}{\tan(\pi/(2\alpha))}\frac{b-a}{b+a}
\iff 2\arctan\frac{b-a}{(b+a)\tan(\pi/(2\alpha))} \geq \theta. $$ Therefore, $S(b/a)$ is included in
$\RS\setminus\Omega_{\alpha}$ if and only if
\begin{equation} \label{eq-sup-close2}
\left(\tan\frac{\pi}{2\alpha}\right)^{-1} \geq
\frac{\tan(\frac{3}{2}\arcsin\sigma)}{\sigma},\quad\sigma=\frac{b/a-1}{b/a+1}\in(0,1).
\end{equation}
The RHS of (\ref{eq-sup-close2}) tends to $3/2$ as $\sigma\to 0$. So, if we take
$\alpha>\pi/(2\arctan(2/3))\approx2.67$, we may find $\delta_0>0$ such that as soon as
$\log(b/a)=\delta_C(x,y)\leq\delta_0$, inequality (\ref{eq-sup-close2}) holds. Then
$\Omega_{\alpha}\subset\mathring{L}(x,y)$ and the lemma follows from (\ref{eq-capdisks}).
\end{proof}

\begin{proposition} \label{prop-sup}
We have for all $x$, $y\in C$, $ d_C(x,y)\leq \pi\sqrt{2}\exp(\delta_C(x,y)/2)$.
\end{proposition}

\begin{figure}[htb]
\begin{center}
\begin{tabular}{cc}
\begin{minipage}{4.1cm} \input{regionA.pstex_t} \end{minipage} &
\begin{minipage}{9.4cm} \input{Region0.pstex_t} \end{minipage} \\
\begin{minipage}{4.1cm} \caption{Region $A$.}\label{fig-regionA} \end{minipage} &
\begin{minipage}{9.4cm} \caption{Region $\Omega$.}\label{fig-regionO} \end{minipage}
\end{tabular}
\end{center}
\end{figure}

\begin{proof}
We may assume that $\delta_C(x,y)<\infty$ and that $x$ and $y$ are linearly independent. Denote
$a=\inf|E(x,y)|>0$, and $b=\sup|E(x,y)|<\infty$. Since $b<\infty$, all $D_{m,l}$,
$(m,l)\in\mathcal{F}(x,y)$, are open finite disks. By Lemma \ref{lem-convexhull}, the closed convex hull of
the centers is included in $\overline{E(x,y)}$. So this convex set does not intersect the open disk
$D(0,a)$ of radius $a$ centered at $0$. It may therefore be separated from $D(0,a)$ by an affine (real)
line. Up to a rotation (or equivalently, up to replacing $y$ by $\alpha y$, $\alpha\in\C$, $|\alpha|=1$), we
may assume that this line has equation $\Re(z)=a$. Now, let $D(c,r)$ be a disk included in $E(x,y)$ whose
center $c$ satisfies $\Re(c)\geq a$. Since the disk $D(c,r)$ must also be included in the closed annulus
$\{z:\:a\leq|z|\leq b\}$, we see that $D(c,r)$ is included in the closed region $A$ bounded by the right
half circles of center $0$ and radius $a$ and $b$, the left half circle of center $i(b+a)/2$ and radius
$(b-a)/2$ and the left half circle of center $-i(b+a)/2$ and radius $(b-a)/2$ (see figure
\ref{fig-regionA}).

We consider now the Möbius transformation
\begin{equation} \label{eq-sup-mob}
h(z) = \frac{z+\sqrt{ab}}{z-\sqrt{ab}}.
\end{equation}
We will denote $\tau=b/a>1$. Then we have the following (where
$D(c,r)$ denotes the open disk of center $c$ and radius $r$):
\begin{itemize}
\item[-] $h$ maps $D(0,a)$ onto $D\left(-\dfrac{\tau +1}{\tau -1},\dfrac{2\sqrt{\tau}}{\tau -1}\right)$.
\item[-] $h$ maps $\RS\setminus \overline{D}(0,b)$ onto $D\left(+\dfrac{\tau +1}{\tau
-1},\dfrac{2\sqrt{\tau}}{\tau -1}\right)$.
\item[-] $h$ maps $\overline{D}\left(+i\dfrac{b+a}{2},\dfrac{b-a}{2}\right)$ onto
$\overline{D}\left(-i\dfrac{\tau+1}{2\sqrt{\tau}},\dfrac{\tau -1}{2\sqrt{\tau}}\right)$.
\item[-] $h$ maps $\overline{D}\left(-i\dfrac{b+a}{2},\dfrac{b-a}{2}\right)$ onto
$\overline{D}\left(+i\dfrac{\tau+1}{2\sqrt{\tau}},\dfrac{\tau -1}{2\sqrt{\tau}}\right)$.
\item[-] $h$ maps respectively $0$, $\infty$ to $-1$, $1$, and the left half-plane to the unit disk.
\end{itemize}
Thus $\RS\setminus A$ is mapped by $h$ onto an open set $\Omega\subset\C$ as in figure \ref{fig-regionO}.
If $t\in [-1,1]$, then 
the distance $\rho(t)$ from $t$ to the boundary of
$\Omega$ is given by
$$
\rho(t)=\left|t+i\frac{\tau+1}{2\sqrt{\tau}}\right|-\frac{\tau-1}{2\sqrt{\tau}}=\sqrt{t^2+\dfrac{(\tau+1)^2}
{ 4\tau } } - \frac{\tau-1}{2\sqrt{\tau}}.$$
If $p(z)|dz|$ denotes the Poincaré metric on $\Omega$, then $p(z)\leq 2/\rho(z)$ (see \eg\ \cite{Mil06},
p223). Therefore,
\begin{eqnarray}
d_C(x,y)&\leq&d_{\RS\setminus A}(0,\infty) = d_{\Omega}(-1,1) \nonumber \\
 &\leq& \int_{-1}^1 \frac{2dt}{\rho(t)}= 
 2 \int_{-1}^1 \frac{\sqrt{t^2+\frac{(\tau+1)^2}{4\tau}} +\frac{\tau-1}{2\sqrt{\tau}}}{t^2+1} dt \nonumber\\
&\leq& 2 \int_{-1}^1 \frac{\sqrt{2}\frac{\tau+1}{2\sqrt{\tau}} + \frac{\tau-1}{2\sqrt{\tau}}}{t^2+1} dt
\leq \pi\sqrt{2\tau} \nonumber\\
&=&\pi\sqrt{2}\exp(\delta_C(x,y)/2) \label{eq-sup-1}
\end{eqnarray}
Hence the result.
\end{proof}

\begin{remark} \label{remark-estimate}
One cannot find a simpler bound in Proposition \ref{prop-sup}. Indeed, let consider the following sequences
of elements of $\C_+^3$:
\begin{eqnarray*}
x_k &=&
\left(1,e^{i\left(\pi/2-\pi/(2k)\right)},e^{i\left(\pi/2-\pi/(2k)\right)}
\right),\\
y_k &=&
\left(2,e^{i\left(\pi/2-\pi/(2k)\right)},e^{i\left(\pi/2-\pi/(2k)\right)}
+2i\cos\frac {\pi}{2k}\right), \\
z_k &=&
\left(\frac{2}{\sqrt{3}\cos\frac{\pi}{2k} +\sin\frac{\pi}{2k}},1,1\right).
\end{eqnarray*}
Using similar techniques, one shows that $d_{\C_+^3}(x_k,y_k)\geq k\log 2$,
and that $d_{\C_+^3}(z_k,x_k)$ and $d_{\C_+^3}(z_k,y_k)$ are $O(\log k)$. Therefore,
$(d_{\C_+^3}(x_k,z_k) +d_{\C_+^3}(z_k,y_k))/d_{\C_+^3}(x_k,y_k) \to 0$. This is because, in
$E_{\C_+^3}(x_k,y_k)$, there are two disks intersecting making a very small angle, $\pi/(2k)$, and thus are
almost tangent; whereas in $E_{\C_+^3}(z_k,y_k)$ and $E_{\C_+^3}(z_k,x_k)$ the angles do not tend to $0$.

Finally, consider the $3\times 3$ matrix $A_k=(a_{ij})$ where $a_{ii}=1$ and $a_{ij}=\alpha/k$ for 
$i\neq j$, $\alpha>0$. Theorem \ref{th-estimate} shows that $\delta_{\C_+^3}\textrm{-diam } A(\C_+^3\MZ) = O(log k)$.
One also checks that for $0<\alpha<\pi/16$ and $k$ large enough, one has $A_k^{-1} x_k$,
$A_k^{-1} y_k \in \C_+^3$ ($x_k$, $y_k$ as above), so that $d_{\C_+^3}\textrm{-diam } A(\C_+^3\MZ)
\geq d_{\C_+^3}(x_k,y_k) \geq k\log 2$.
\end{remark}

\begin{theorem} \label{th-diameter}
Let $C$ satisfy (C1)-(C2). Let $C_1\subset C$ be any complex subcone of
$C$. Consider the three diameters of $C_1$ with respect to $\delta_C$, $d_C$, $\td_C$. If any of these
diameters is finite, then the two others are also finite.
\end{theorem}

\begin{proof}
This comes immediately from Propositions \ref{prop-sup} and \ref{prop-inf}.
\end{proof}

\begin{theorem} \label{th-contraction-td}
Let $C_1\subset V_1$, $C_2\subset V_2$ satisfy (C1)-(C2). Let $T:V_1 \to V_2$ be a linear map, and suppose
that $T(C_1\MZ)\subset C_2\MZ$. If the diameter $\tilde{\Delta}=\sup \td_{C_2} (Tx,Ty)$ is
finite, then we have
$$ \forall x,\,y\in C_1,\quad \td_{C_2}(Tx,Ty) \leq
\tanh\left(\frac{\pi\exp(\tilde{\Delta})}{2\sqrt{2}}\right)
\td_{C_1}(x,y).$$
\end{theorem}

\begin{proof}
Denote by $\Delta$ the diameter of $T(C_1\MZ)$ with respect to $d_C$. Then from Propositions
\ref{prop-sup} and \ref{cor-tilde}, we have $\Delta\leq \pi\sqrt{2}\exp(\tilde{\Delta})$. The conclusion then
follows from Theorem \ref{th-rugh} and the definition of $\td_C$.
\end{proof}

\end{document}